\documentclass[11pt]{article}
\usepackage{surj-commands}

\title{On the Surjectivity of Galois Representations Associated to Elliptic Curves over Number Fields}
\author{Eric Larson and Dmitry Vaintrob}
\date{}

\begin{document}
\maketitle

\begin{abstract}
Given an elliptic curve $E$ over a number field $K$, the $\ell$-torsion
points $E[\ell]$ of $E$
define a Galois representation $\gal(\bar{K}/K)\to \gl_2(\ff_\ell)$.
A famous theorem of
Serre \cite{serre} states that as long as
$E$ has no Complex Multiplication (CM), the
map $\gal(\bar{K}/K)\to \gl_2(\ff_\ell)$ is surjective for all but
finitely many $\ell$.

We say that a
prime number $\ell$ is \emph{exceptional} (relative to the pair $(E,K)$) if
this map is \emph{not} surjective. Here we give a new bound on 
the largest exceptional prime, as well as on the product of
all exceptional primes of $E$. We show in particular that
conditionally on the Generalized Riemann Hypothesis (GRH),
the largest exceptional prime of an elliptic curve $E$
without CM is no larger
than a constant (depending on~$K$) times $\log N_E$,
where $N_E$ is the absolute value of the norm of the conductor.
This answers affirmatively a question of Serre in \cite{quelques}.
\end{abstract}

\section{Introduction}
Let $E$ be an elliptic curve over a number field $K$, and for
each prime number $\ell$, let $E[\ell]$ be the group of $\ell$-torsion
points of $E$ over $\bar{K}$. This group is isomorphic to 
$\left(\zz/\ell\zz\right)^2$ and has action by the absolute Galois group
$G_K: = \gal(\bar{K}/K)$, which we denote
\[\rho_{E,\ell}\colon G_K \to \gl(E[\ell]) \simeq \gl_2(\ff_\ell).\]
The collection of representations $\rho_{E,\ell}$ encode many important properties
of $E$, such as its primes of bad reduction
and its number of points over finite fields.

As long as $E$ has no complex multiplication (CM),
these representations are surjective for all but finitely
many $\ell$, which we call \emph{exceptional primes}
for $E$. This result was proven in Serre's 1968 paper \cite{serre}, and
concluded the proof of the long-conjectured Open Image Theorem --- the
statement that the inverse limit of the images
\[\varprojlim_{m\in\zz}\rho_{E,m}(G_K)\subset \varprojlim_m\gl_2(\zz/m\zz)\]
has finite index in 
$\underset{m}{\varprojlim}\gl_2(\zz/m\zz) \cong \gl_2(\hat{\zz})$.

Serre's original proof was ineffective, even over the ground field $\qq$. 
But in the later paper \cite{quelques}, he gave in the case
of $K=\qq$ an explicit upper bound
on the largest exceptional prime of an elliptic curve $E$ over
the rational numbers without CM, conditionally
on the Generalized Riemann Hypothesis (GRH).
Namely he showed that the largest
exceptional prime $\ell_E$ is
bounded by the following expression in the conductor $N_E$
of the elliptic curve:
\begin{align}\label{serreq}
\ell_E \leq  C_1 \cdot \log N_E \cdot (\log\log N_E)^3, 
\end{align}
for $C_1$ an absolute (and effectively computable) constant.
In the same paper, he conjectured
that, conditionally on GRH, a similar bound should hold 
for elliptic curves defined over arbitrary fields $K$. 

An effective bound over arbitrary number fields $K$ was later given,
unconditionally, by the paper of Masser and 
W\"ustholz \cite{mw}, with bound $C_2 \cdot \max(h_E, n_K)^\gamma$
for absolute constants $C_2$ and $\gamma$,
where $h_E$ is the logarithmic height of the $j$-invariant
of $E$ and $n_K$ is the degree of $K$.
Here, the constant $\gamma$ is very large (although it can
be reduced to $2$ if we only care about bounding degrees of isogenies).
Our results imply that conditionally on GRH, we can take $\gamma = 1$
if we let $C_2$ depend on $K$.

Over $\qq$, Kraus and Cojocaru (\cite{kraus} and \cite{cojocaru})
gave another unconditional bound in terms of the conductor
using the modularity of elliptic curves over $\qq$, namely
\[\ell_E \leq C_3 \cdot N_E \cdot (\log \log N_E)^{1/2}.\]
Moreover,
in \cite{zywina}, Zywina shows that the product
\[A_E:=\prod_{\ell\text{ exceptional for }E}\ell\]
can be bounded by the $b_E$th power of each of the
above bounds on $\ell_E$, where $b_E$ is the number of
primes of bad reduction for $E$.

The gradual improvements in the bound on exceptional primes
have paid off. 
A recent paper of Bilu and Parent \cite{bilu}
which made a breakthrough in the search for a uniform
bound on exceptional primes over $\qq$ 
(showing that $X_{\text{split}}(\ell)(\qq)$
consists only of CM points and cusps for $\ell$ sufficiently large)
relied crucially the value of $\gamma$
appearing in the Masser-W\"ustholz bound.

This paper continues this tradition. We bound, conditionally
on GRH, both the largest exceptional prime $\ell_E$ and the product
of all exceptional primes $A_E$. Our proof is in the spirit of Serre's
original bound in \cite{quelques}, but we allow 
$E$ to be defined over an arbitrary number field $K$,
which entails a more delicate analysis.
The bound on the largest exceptional prime 
we get is, as conjectured in 
\cite{quelques}, the same as what Serre
obtained when $K = \qq$ (equation \eqref{serreq},
with the constant $C_1$ replaced by a constant $C(K)$ 
depending on the number field $K$).

In fact, for fixed ground field $K$,
we show that an asymptotically
better bound holds. Namely, conditionally on GRH, the
largest exceptional prime $\ell_E$ satisfies
\[\ell_E\le C'(K) \cdot \log N_E,\]
where $N_E$ is the absolute value of the norm of the conductor of $E$.

We make the constant $C(K)$ in our first bound effective,
but have at the moment no
effective way of determining the constant $C'(K)$ in
the second, asymptotically better bound (even over $K=\qq$).

We also give a conditional bound on the product of all exceptional 
primes, $A_E$. We show, in particular, that for fixed $K$ and
fixed $\epsilon > 0$,
we have 
$A_E<N_E^{\epsilon}$ for all but finitely many curves $E$. 
The bound one would get
by multiplying together all primes up to our upper bound for $\ell_E$ --- as
well as bounds on $A_E$ given in earlier papers --- 
give values which are asymptotic to a positive power of $N_E$.

\medskip

\textit{\textbf{For the remainder of the paper, we assume the Generalized
Riemann Hypothesis (GRH).}}

\medskip

Our proof can be roughly outlined as follows. 
First we compare an exceptional prime $\ell$
and an \emph{unexceptional} prime $p$, and show that the two
Galois representations $\rho_{E,\ell}$ and $\rho_{E,p}$ impose
conditions on traces of Frobenius of $E$ which are incompatible
if $\ell$ is sufficiently large compared to $p$ and $N_E$. This part relies on
the effective Chebotarev Theorem of Lagarias and Odlyzko together with
a result of our earlier paper \cite{ourfirstpaper}.

Next, we give an upper bound for the \emph{smallest unexceptional}
prime $p$. The analysis here bifurcates. Ineffectively,
it can be easily shown that the smallest such $p$ is bounded above
by a constant depending only on $K$. The effective
bound is trickier, and uses Serre's method in \cite{quelques},
which depends on GRH in an essential way.

Combining the bound on the unexceptional $p$ 
with the bound on the exceptional $\ell$ in
terms of $p$ completes the proof. We then show that 
the bound on $\ell$ can be tweaked to give an upper bound
on the product $A_E$ of all exceptional primes.
Throughout the
paper, we treat separately two different kinds of exceptional
primes: those for which 
$\rho_{E,\ell}$ is absolutely irreducible, and those for which
it is not. While the analysis in the two cases
is remarkably parallel, our bound on the product of
exceptional primes $\ell$ of the second kind (such that
$\rho_{E,\ell}$ is reducible over $\bar{\ff}_\ell$) turns out to
be significantly better, polynomial in $\log N_E$
(see Lemma \ref{lm:redone}).


\medskip

Fix a number field $K$, and write
$n_K$, $r_K$, $R_K$, $h_K$, and $\Delta_K$ for the degree,
rank of the unit group, regulator, class number, and discriminant of $K$
respectively.
Let us choose for every prime ideal $v$ of $K$,
a corresponding Frobenius element $\pi_v \in G_K := \gal(\bar{K}/K)$.
We let $E$ be an elliptic curve \emph{without complex multiplication (CM)},
and we write $N_E$ and $a_E$
for the absolute value of the norm of the conductor of $E$, and
the number of primes of additive reduction of $E$, respectively.
We say that $X \llk Y$ if there are \emph{effectively computable}
constants $A$ and $B$
depending only on $K$ for which $X \leq AY + B$.
Moreover, we say that $X \lllk Y$
if $X \leq AY + B$ for constants $A$ and $B$ that
are \emph{not} assumed to be effectively computable.
If the constants $A$ and $B$ are absolute, we drop the $K$
subscript on the $\ll$ and $\lll$.
With this notation, our results are as follows.

\begin{thm}[Theorem~\ref{ineffective}] \label{ithm:ineffective}
Let $E$ be an elliptic curve over a number field $K$ without CM.
Then any exceptional prime $\ell$ satisfies
\[\ell \lllk \log N_E.\]
Moreover, the product of all exceptional primes satisfies
\[\prod \ell \lllk 4^{a_E} \cdot (\log N_E)^{14}.\]
\end{thm}

\begin{thm}[Theorem~\ref{effective}] \label{ithm:effective}
Let $E$ be an elliptic curve over a number field $K$ without CM.
Then any exceptional prime $\ell$ satisfies
\[\ell \llk \log N_E \cdot (\log \log N_E)^3.\]
Moreover, the product of all exceptional primes satisfies
\[\prod \ell \llk 4^{a_E} \cdot (\log N_E)^{14} \cdot (a_E + \log \log N_E)^6 \cdot (\log \log N_E)^{36}\llk 4^{a_E} \cdot (\log N_E)^{21},\]
where $a_E$ is the number of primes of $K$ of additive reduction for $E$.
\end{thm}

\subsection*{Acknowledgements}
We would like to thank David Zureick-Brown,
Bryden Cais and Ken Ono for suggesting
us the problem that led to this paper and answering questions.
Thanks also to Barry Mazur for valuable comments and discussions.
We would also like to express our admiration for
Serre's papers
\cite{serre} and \cite{quelques},
whose techniques provide the mathematical inspiration and
departure-point of our work.

\section{\boldmath Possible Images of the Representation $\rho_{E, \ell}$ \label{sec:images}}

In this section, we analyze the possible images
of $\rho_{E, \ell}$.
The proofs of all of the results of this section are in
the papers \cite{serre} and \cite{quelques} by Serre. 
We begin by singling out some subgroups of $\gl_2(\ff_\ell).$

\begin{defi}\label{def:cartan}
A \emph{Cartan} subgroup is a subgroup of
$\gl_2(\ff_\ell)\subset \gl_2(\bar{\ff}_\ell)$ 
which fixes two one-dimensional subspaces
of $\bar{\ff}_\ell^2$, i.e.\ which
in some basis of $\bar{\ff}_\ell^2$ looks like
\[\left(\begin{array}{cc}
* & 0 \\
0 & *
\end{array}\right).\]
\end{defi}

A Cartan subgroup is index two in its normalizer.
The normalizer consists of matrices of the form
\[\left(\begin{array}{cc}
* & 0 \\
0 & *
\end{array}\right)
\quad \text{or} \quad
\left(\begin{array}{cc}
0 & * \\
* & 0
\end{array}\right)\]
(i.e. matrices which either fix or permute the two
subspaces fixed by the Cartan subgroup).

\begin{lm} \label{subgroups}
Let $G$ be any subgroup of $\gl_2(\ff_\ell)$.
Then, one of the following holds:
\begin{enumerate}
\item \label{sg:bor}
(Reducible Case)
$G$ acts reducibly on $\bar{\ff}_\ell^2$.
\item \label{sg:nc}
(Normalizer Case)
$G$ is contained in the normalizer of
a Cartan subgroup,
but not in the Cartan subgroup itself.
\item \label{sg:sl}
(Special Linear Case)
$G$ contains $\slg_2(\ff_\ell)$.
\item \label{sg:exc}
(Irregular Case)
The image of $G$ under
the projection $\gl_2(\ff_\ell) \to \pgl_2(\ff_\ell)$,
is contained in a subgroup which
is isomorphic to $A_4$, $S_4$, or $A_5$.
\end{enumerate}
\begin{rem}
We use the term ``irregular'' subgroup to avoid a clash
of notation; usually they are called ``exceptional''
subgroups.
\end{rem}

\end{lm}
\begin{proof} See Section~2 of \cite{serre}.
\end{proof}

\begin{defi} Having fixed the field $K$, 
we call a prime number $p$ \emph{acceptable}
if $p$ is unramified in $K/\qq$
and $p\ge 53$. (So almost all primes are acceptable.)
For the remainder of the paper, we will
only consider acceptable primes.
\end{defi}

\begin{lm} \label{bigproj} If $p$ is acceptable, then
$\prho_{E, p}$ contains an element of order
at least $13$.
\end{lm}
\begin{proof}
This follows from Lemma~18$'$ of \cite{quelques}
(which is stated for $K = \qq$, but the same proof works
as long as $p$ is unramified in $K$).
\end{proof}

\begin{lm} \label{rednorm}
Let $\ell$ be an acceptable exceptional prime.
Then the image of $\rho_{E, \ell}$ falls into either the
reducible case or the normalizer case of Lemma~\ref{subgroups}.
\end{lm}

\begin{proof}
Since $\ell \nmid \Delta_K$, it follows that
$\det \rho_{E, \ell}$ is surjective, so
the image of $\rho_{E, \ell}$ cannot fall into case~\ref{sg:sl}
because $\ell$ is exceptional.
By Lemma~\ref{bigproj},
the image of $\rho_{E, \ell}$ cannot fall into case~\ref{sg:exc}.
\end{proof}

The two remaining cases will require separate analysis, and throughout
the paper we will separate them as the ``reducible'' case and the
``normalizer'' case. 

\section{The Effective Chebotarev Theorem \label{sec:cheb}}

We have the following effective version
of the Chebotarev Density Theorem, due to
Lagarias and Odlyzko.

\begin{thm}[Effective Chebotarev]\label{thm:chebotarev}
Let $L/K$ be a Galois extension of number fields
with $L \neq \qq$.
Then every conjugacy class of $\gal(L/K)$ is represented
by the Frobenius element
of a prime ideal $v$ such that
\[\nm^K_\qq(v) \ll (\log \Delta_E)^2.\]
\end{thm}
\begin{proof}
See \cite{lo}, remark at end of paper
regarding the improvement to Corollary~1.2.
\end{proof}

\begin{cor}[Effective Chebotarev with avoidance]\label{cor:ceb}
Let $L/K$ be a Galois extension of number fields
with $L \neq \qq$ and $\Sigma\subset \Sigma_K$ a finite set of primes
which includes the
primes at which $L/K$ is ramified.
Let $N$ be the norm of the product of the primes of $\Sigma$,
and write $d = [L:K]$.
Then every conjugacy class of $\gal(L/K)$ is represented
by the Frobenius element
of a prime ideal $v \in \Sigma_K \smallsetminus \Sigma$ such that
\[\nm^K_\qq(v) \ll d^2 \cdot \big(\log N + \log \Delta_K + n_K \log d \big)^2 \llk d^2 \cdot \big(\log N + \log d\big)^2. \]
\end{cor}
\begin{proof}
Let $H$ be the Hilbert class field of $K$, of degree $h_K$ over $K$. Then 
$\Delta_H =\Delta_K^{h_K}$, so any element of the class group
is represented by a prime ideal $v\in \Sigma_K$ of
norm $\ll (h_K\log\Delta_K)^2$.
It follows from a result of Lenstra
(Theorem~6.5
in \cite{lenstra}) that $h_K \leq \Delta_K^{3/2}$,
so we can take $\nm(v)\ll \Delta_K^4.$
Now, we let $I = \prod_{v \in \Sigma} v$,
and apply this result to the image in the
class group of the ideal $I^{-1}$.
We get a prime ideal $v_0$ with $\nm(v_0) \ll \Delta_K^4$
such that $v_0I$ is principal, generated by $x\in K$.

Define $L' = L[\sqrt[3]{x}, \omega]$,
for a primitive cube root of unity $\omega$.
The set $\Sigma' \subset \Sigma_K$ of primes ramified
in $L'/K$ consists of
all elements of $\Sigma$, plus some primes
dividing $6v_0$.

Now, we apply effective Chebotarev again, to $\gal(L' / K)$,
to conclude that every conjugacy class of $\gal(L'/K)$
is represented by a Frobenius element of a
prime ideal $v \in \Sigma_K$ which is unramified
in $L'$, and thus not in $\Sigma$, with
\[\nm^K_\qq(v) \ll (\log \Delta_{L'})^2.\]
We now turn to bounding $\log \Delta_{L'}$. For a prime $v$ of $K$,
write $e_v$ and $f_v$ for the ramification and inertial degrees of $v$
respectively. We have
\begin{align*}
\log \Delta_{L'} &= [L':K] \cdot \log \Delta_K + \log \nm^{L'}_K \mathfrak{d}^{L'}_K \\
&\leq 6d \log \Delta_K + \sum_{v \in \Sigma'} \left((6d - 1) f_v \log p_v + 6df_v e_v \text{val}_{p_v}(d) \log p_v \right)\\
&\leq 6d \cdot \left(\log \Delta_K + \sum_{v \in \Sigma'} f_v \log p_v + \sum_{v \in \Sigma'} f_v e_v \text{val}_{p_v}(d) \log p_v \right) \\
&\leq 6d \cdot \left(\log \Delta_K + \log (N \cdot 6 \Delta_K^4) + n_K \log d \right) \\
&\ll d \cdot \left(\log N + \log \Delta_K + n_K \log d \right). && \qedhere
\end{align*}
\end{proof}

Throughout the paper, we will frequently apply the above corollary
to Galois representations built out of the representations
$\rho_{E, \ell}$. For this purpose, recall the well-known
N\'eron-Ogg-Shafarevich criterion:

\begin{thm}[N\'eron-Ogg-Shafarevich]
Let $E$ be an elliptic curve over $K$. Then $\rho_{E, \ell}$ is ramified
only at primes dividing $\ell$ and the conductor of $E$.
\end{thm}
\begin{proof}
This is well known; see e.g.\ Proposition~4.1 of \cite{silverman}.
\end{proof}

\section{Bounds In Terms of The Smallest Unexceptional Prime \label{sec:ub}}

Recall that we have fixed an elliptic curve $E$ over a number
field $K$, and $\ell$ is an exceptional prime for $(E, K)$.
In this section we give bounds on both the largest
exceptional prime and the product of all exceptional primes,
in terms of the smallest \emph{unexceptional} prime.

\subsection{The Reducible Case}

Suppose that $E[\ell]$ is reducible over $\bar{\ff}_\ell$,
and write
\[\rho_{E, \ell} \otimes_{\ff_\ell} \bar{\ff}_\ell = \left(\begin{array}{cc}
\psi_\ell^{(1)} & - \\
0 & \psi_\ell^{(2)}
\end{array}\right).\]

\begin{thm} \label{paperonethm}
There exists a finite set $S_K$ of
primes numbers depending only on $K$ such that if
$\ell \notin S_K$, then there exists a CM
elliptic curve $E'$, which is defined over $K$
and whose CM-field is contained in $K$, such that
for some character $\epsilon_\ell \colon \gal(\bar{K} / K) \to \mu_{12}$,
\begin{align}\label{paperoneformula}
\begin{cases}
\psi_\ell^{(1)} &= \phi_\ell^{(1)} \otimes \epsilon_\ell \\
\psi_\ell^{(2)} &= \phi_\ell^{(2)} \otimes \epsilon_\ell^{-1}
\end{cases} \quad \text{where} \quad
\rho_{E', \ell} \otimes_{\ff_\ell} \bar{\ff}_\ell = \left(\begin{array}{cc}
\phi_\ell^{(1)} & 0 \\
0 & \phi_\ell^{(2)}
\end{array}\right).
\end{align}
Moreover the elliptic curve $E'$ depends only on $E$
(i.e.\ is independent of $\ell$), and $\epsilon_\ell$
is ramified only at primes dividing the conductor of $E$.
\end{thm}
\begin{proof} See Theorem~1 of \cite{ourfirstpaper}
and Remark~1.1 following the theorem.
\end{proof}
This lets us relate the Frobenius polynomials
of $E$ and $E'$ at small primes of $K$.
We make the following definitions.
\begin{defi}\label{def:A}
Fix $E$ and $E'$ as above. We define 
$R_E$ to be the product of all reducible primes $\ell$ satisfying 
equation~\eqref{paperoneformula}.
\end{defi}

\noindent
The fact that $E'$ depends only on $E$ (for $\ell\ggk 1$)
implies that 
\[\prod_{\rho_{E,\ell}\text{ reducible }}\ell \llk R_E.\]
(Moreover, this is sharp, as $R_E$ divides the product on the left.)
\begin{defi}
  For a polynomial $P\in \zz[x],$ define its \emph{$12$th Adams operation}
$\Psi^{12}P\in \zz[x]$ to be the polynomial whose roots 
(in $\cc$) are the twelfth
powers of the roots of $P$. 
\end{defi}

\noindent
Using this notation and writing 
\[P_E(v) = x^2 + \tr_E(\pi_v) x + \nm(v)\]
for the Frobenius
polynomial of $\pi_v \in G_K$, we have the following result
(where $E'$ is the CM elliptic curve from above).

\begin{lm}\label{lm:frobenius}
Let $v$ be a prime of $K$ at which $E$ has good reduction.
If $4 (\nm v)^6 < R_E$, then 
\[\Psi^{12}P_E(v) = \Psi^{12}P_{E'}(v);\] 
moreover, if $\ell \mid R_E$ is
such that $4\sqrt{\nm v} < \ell$
and $\epsilon_\ell(\pi_v) = 1$ (where
$\epsilon_\ell \colon G_K \to \mu_{12}$ is as in Theorem~\ref{paperonethm}),
then \[P_E(v)=P_{E'}(v).\]
\end{lm}
\begin{proof}
Suppose $\ell \mid R_E$, i.e.\ $\ell$ satisfies equation~\eqref{paperoneformula}.
In particular, $(\psi_\ell^{(1)})^{12} = (\phi_\ell^{(1)})^{12}$ and $(\psi_\ell^{(2)})^{12} = (\phi_\ell^{(2)})^{12}$,
i.e.\ $\Psi^{12}P_E\equiv\Psi^{12}P_{E'}$ mod $\ell$.
Since this holds for all $\ell \mid R_E$, by plugging in $v$ we obtain
\[\Psi^{12}P_E(v) \equiv \Psi^{12}P_{E'}(v) \mod R_E.\]
If moreover $\epsilon_\ell(\pi_v) = 1$, then
$\psi_\ell^{(1)}(\pi_v) = \phi_\ell^{(1)}(\pi_v)$
and $\psi_\ell^{(2)}(\pi_v) = \phi_\ell^{(2)}(\pi_v)$.
Equivalently,
\[P_E(v) \equiv P_{E'}(v) \mod \ell.\]

From the Weil bounds,
$P_{E_0}(v)$ has nonpositive discriminant
and constant term $\nm v$ for any elliptic curve
$E_0$ and prime $v$ of good reduction for $E_0$.
In other words,
\[P_{E_0}(v)=x^2 - ax + \nm v \quad \text{and} \quad \Psi^{12} P_{E_0}(v) = x^2 - bx + \nm v^{12},\]
with $|a|\le 2\sqrt{\nm v}$ and $|b|\le 2(\nm v)^6$.
It follows that
$P_E(v)-P_{E'}(v)=Ax$ for some $|A|\le 4\sqrt{\nm v}$
and $\Psi^{12}P_E-\Psi^{12}P_{E'}=Bx$ for some 
$|B|\le 4(\nm v)^6$.
On the other hand, we have seen above that $\ell\mid A$ and $R_E \mid B$.
The lemma follows, using that $|A|<\ell$ and $\ell\mid A$ imply $A=0$
(and similarly for $B$).
\end{proof}

Now we are in a position to bound any prime $\ell$ with
reducible $\rho_{E,\ell}$ (or the product of all such)
in terms of a small unexceptional prime $p$.

\begin{lm} \label{gth}
Suppose that $p$ is an acceptable prime that does not divide $R_E$.
Let $E'$ be as above,
and let $H \subset \gl_2(\ff_p) \times \gl_2(\ff_p)$
be the image of $\rho_{E, p} \times \rho_{E', p}$.
Then there exists a surjection $f \colon H \twoheadrightarrow G$
with $|G| \ll p^3$,
and a $g \in G$ such that for any $(X, Y) \in H$ with $f(X, Y) = g$,
we have $\tr(X^{12}) \neq \tr(Y^{12})$.
\end{lm}
\begin{proof}
First suppose that $p$ is unexceptional.
By the theory of complex multiplication,
the image of
$\rho_{E', p}$ is contained in either a split or a nonsplit
Cartan subgroup.
Hence, the image of the projectivization $\prho_{E', p}$
is contained in a cyclic group of order $p \pm 1$.
Since $p$ is acceptable, $p \pm 1 \nmid 12$.
It follows that there is
an $M \in \pgl_2(\ff_p)$ whose $12$th power is not conjugate
to anything in the image of $\prho_{E', p}$.
Taking $G = \pgl_2(\ff_p)$ and $f$ to be projection
onto the first factor followed by the projection
$\gl_2(\ff_p) \to \pgl_2(\ff_p)$ completes the proof
in this case.

Hence we can assume that $p$ is of normalizer type, or of
reducible type but not satisfying the condition (\ref{paperoneformula})
with respect to the CM curve $E'$.
Write $\tilde{\rho}_{E, p}$
for the semisimplification (i.e.\ direct sum of the
Jordan-Holder quotients) of $\rho_{E, p}$. 
To bound the product of all exceptional primes
of $E$, we consider the Galois representation
\[\Pi = \tilde{\rho}_{E, p} \times \rho_{E', p} \colon G_K \to \gl_2(\ff_p) \times \gl_2(\ff_p).\]
Since $\det \rho_{E, p} = \det \rho_{E', p}$ is surjective onto $\ff_p^\times$
and either Cartan or Normalizer
subgroups of $\gl_2(\ff_p)$ have $\ll p^2$ elements,
the image has order $\ll p^3$.

Now we claim that the image of $\Pi$ contains
something of the form $(X, Y)$ for which $\tr X^{12} \neq \tr Y^{12}$.
If $p$ is of reducible type, then this is clear
by the assumption that $p \nmid R_E$.
If $p$ is of normalizer type, then
by Lemma~\ref{bigproj},
the projective
image of $\prho_{E, p}$ contains an element of order
at least $13$.
In particular, it must contain something of the form
\[A = \left(\begin{array}{cc}
a & 0 \\
0 & b
\end{array}\right),\]
with $a^{12} \neq b^{12}$. Let $B$ be an element of 
$\gl_2(\ff_p)$ in the image of $\rho_{E, p}$ that
is not in the Cartan group.
Since the image of $\rho_{E', p}$ is abelian,
it follows that the image of $\Pi$ contains
$(1, M)$, where $M = ABA^{-1}B^{-1}$. By explicit computation,
\[M = \left(\begin{array}{cc}
a^{-1}b & 0 \\
0 & b^{-1}a
\end{array}\right).\]
Taking $X = 1$ and $Y = M$ thus completes the proof.
\end{proof}

\begin{lm} \label{lm:redone}
Let $p$ be the smallest acceptable prime that does not divide $R_E$.
\[R_E \llk p^{36} \cdot (\log N_E + \log p)^{12}.\]
Moreover, any prime $\ell \mid R_E$ satisfies
\[\ell \llk p^3 \cdot (\log N_E + \log p).\]
\end{lm}

\begin{proof}
Let $f\colon H \twoheadrightarrow G$ and $g\in G$ be as in
Lemma~\ref{gth}.

First, we bound $R_E$.
By Corollary~\ref{cor:ceb} applied to $g \in G$, N\'eron-Ogg-Shafarevich, and Lemma~\ref{gth}, there is
a prime $v$ of good reduction for $E$
such that $\tr \rho_{E, p}(\pi_v^{12}) \neq \tr \rho_{E', p}(\pi_v^{12})$,
which moreover satisfies
\begin{equation}\label{boundonpi}
\nm v \llk p^6 \cdot (\log N_E + \log p)^2.  
\end{equation}
In particular,
$\Psi^{12}P_{E}(v)\neq \Psi^{12}P_{E'}(v)$,
so by Lemma~\ref{lm:frobenius} we have
\[R_E \leq 4(\nm v)^6 \llk p^{36} \cdot (\log N_E + \log p)^{12}.\]


To bound the largest exceptional prime, we consider
the direct sum of $\epsilon_\ell$ and $G$.
Since $\epsilon_\ell$
has order~$12$, the image of this Galois representation
contains $(1, g^{12})$.
Applying Corollary~\ref{cor:ceb} to $g^{12} \in G$, N\'eron-Ogg-Shafarevich,
and Lemma~\ref{gth}, we can find a prime
$v$ of good reduction for $E$
such that $\tr \rho_{E, p}(\pi_v) \neq \tr \rho_{E', p}(\pi_v)$
and $\epsilon_\ell(\pi_v) = 1$, which
satisfies the bound \eqref{boundonpi}.
In particular, $P_E(v) \neq P_{E'}(v)$,
so Lemma \ref{lm:frobenius}
gives $\ell\leq 4\sqrt{\nm v}\llk p^6\cdot (\log N_E + \log p)$,
as desired.
%
\end{proof}

\subsection{The Normalizer Case}

Let $\ell$ be a prime such that
the image of $\rho_{E, \ell}$ falls into the
normalizer case of Lemma~\ref{subgroups}.
Write $C$ for our Cartan subgroup and $N$ for its normalizer.
Then we have a quadratic character $\chi$
on $\gal(\bar{\qq}/K)$ given by

\[\chi \colon \gal(\bar{\qq}/K) \to N \to N/C \simeq \{\pm 1\}.\]

\begin{lm}
The character $\chi$ is ramified only at places of bad additive reduction.
\end{lm}
\begin{proof}
See Lemma~2
in Section~4.2 of \cite{serre}.
\end{proof}

In this case, we say that $\ell$ is
\emph{$\chi$-exceptional (of normalizer type)}.
More generally, 
if $V\subset \hom(G_K,\zz/2)$ is an $\ff_2$-vector space 
of Galois characters,
we say that $\ell$ is $V$-exceptional if $\ell$ is
$\chi$-exceptional for some $\chi \in V$. Note that
the space $V$ of characters induces a Galois extension of
$K$ with Galois group the dual $\ff_2$-vector space $V^*$,
via the following construction.
\begin{defi} For $V\subset \hom(G_K,\zz/2\zz)$,
we write $\rho_V \colon G_K \to V^*$ for the 
map induced by the pairing $V\times G_K^{\text{ab}}\to \ff_2$.
\end{defi}
This gives (functorially) a one-to-one correspondence
between finite $\ff_2$-vector spaces of Galois characters and
finite abelian field extensions with Galois group annihilated
by $2$. 

\begin{lm} \label{lmv}
The vector space $V$ of all quadratic Galois characters
ramified only at places of bad additive reduction satisfies
\[|V| \leq 2^{a_E + 2n_K} \cdot h_K.\]
(In fact, the argument below shows
$|V| \leq 2^{a_E + 2n_K} \cdot 2^{r_2(\cla(K))}$,
where $r_2(\cla(K))$ is the $2$-rank of the class group.)
\end{lm}
\begin{proof}
Note that $|V|=|V^*|.$ 
Write $U_K$ for the subgroup of principal id\`eles in the group
$\mathbb{I}_K$ of id\`eles.
By class field theory, $\rho_V$ induces a surjection
$\mathbb{I}_K \to V^*$. Since $[\mathbb{I}_K : U_K] = h_K$,
it suffices to show that
$\rho_V(U_K)\subset V$ has order at most $2^{a_E + n_K}$.
However, the restriction
$\rho_V|_{U_K}$ factors through the projection
\[U_K\to \prod_{\text{$v$ of additive reduction}} \oo_v^*/(\oo_v^*)^2.\]
Now by a standard application of Hensel's lemma, 
if $p_v\ne 2$ then $\oo_v^* / (\oo_v^*)^2=\ff_2,$ and if $p_v=2$ then
$\oo_v^* / (\oo_v^*)^2$ is a vector space over $\ff_2$
of dimension at most $2e_v f_v$.
Since $\sum_{v \mid 2} 2e_v f_v = 2n_K$,
this gives the desired bound.
\end{proof}

\begin{lm} \label{vexc}
Let $V$ be a $d$-dimensional vector space of quadratic Galois characters
ramified only at places of bad additive reduction,
and let $p$ be the smallest acceptable prime that is not
$V$-exceptional.
Then the product of all $V$-exceptional primes $\ell$
satisfies
\[\prod \ell \llk \left(2^d \cdot p^3 \cdot (\log N_E + \log p)\right)^{2 - 2^{1 - d}} .\]
\end{lm}
\begin{proof}
We start by showing that for any $\alpha\in V^*$,
there is some $X_\alpha\in \pgl_2(\ff_p)$
of nonzero trace such that $(\alpha, X_\alpha)$ is contained in
the image of $\rho_V \times \prho_{E, p}$. 

If $p$ is unexceptional,
then $\prho_{E, p}$ surjects onto $\pgl_2(\ff_p)$.
Hence, the abelianization of $\prho_{E, p}$
is the quadratic character defined by $\pgl_2(\ff_p) / \psl_2(\ff_p)$.
Since $V^*$ is an abelian group, the image of $\rho_V \times \prho_{E, p}$
contains everything of the form $(\alpha, X)$ either for
every $X \in \psl_2(\ff_p)$, or for every $X \notin \psl_2(\ff_p)$.
Either way, the image contains something of the
form $(\alpha, X_\alpha)$ where $X_\alpha$ has nonzero trace.

If $p$ is exceptional, then since $p$ is acceptable,
$p$ is either of normalizer or of reducible type.
Pick some $Y_\alpha$ so that $(\alpha, Y_\alpha)= (\rho_V \times \rho_{E, \ell})(g_\alpha)$
is in the image of $\rho_V \times \prho_{E, p}$.
If $p$ is exceptional of normalizer type, then since
$p$ is not $V$-exceptional, we can choose $Y_\alpha$
so that $Y_\alpha$ lies in the Cartan subgroup. If $\tr(Y_\alpha)\neq 0$,
we are done, so suppose $\tr(Y_\alpha)=0$. 
From Lemma~\ref{bigproj}, there is an element
$Z_\alpha=\prho_{E, p}(h_\alpha)$ of order greater than four
in the image of $\rho_{E,p}$ (which must lie in the Cartan subgroup).
Now we can take $X_\alpha=Y_\alpha Z_\alpha^2$ which has nonzero trace
and satisfies
$(\rho_V \times \rho_{E,p})(g_\alpha h_\alpha^2)=(\alpha,Y_\alpha Z_\alpha^2)$.
as desired.

Now, for each $\alpha \in V^*$, let $X_\alpha$ be the element constructed above.
Applying Corollary~\ref{cor:ceb} and N\'eron-Ogg-Shafarevich,
we can find a prime ideal $v_\alpha$ such that
$(\rho_V \times \epsilon_\ell)(\pi_{v_\alpha}) = (\alpha, X_\alpha)$,
which moreover satisfies
\[\nm v_\alpha \llk 4^d \cdot p^6 \cdot (\log N_E + \log p + d)^2 \llk 4^d \cdot p^6 \cdot (\log N_E + \log p)^2.\]
(The last inequality follows from Lemma~\ref{lmv}, using $a_E \llk \log N_E$.)
This gives by the Weil bound that for any $\alpha\in V^*$ we can
choose $v_\alpha$ so that
\begin{align}\label{eq:normtrace}
0 \neq \tr_E(\pi_{v_\alpha})\llk 2^d \cdot p^3 \cdot (\log N_E+\log p)^2.
\end{align}
Now, $\tr_E(\pi_{v_\alpha})$ must be divisible by all
$V$-exceptional primes $\ell$ whose corresponding
character $\chi_\ell$ satisfies $\chi_\ell(\pi_{v_\alpha}) = -1$.
But for any $\chi_\ell$, half of the $\alpha\in V^*$ satisfy 
$\chi_\ell(\pi_{v_\alpha}) = -1$.
Putting this together,
\[\left. \left(\prod_{\substack{\ell\text{ exceptional}\\
\text{of normalizer type}}} \ell\right)^{2^{d - 1}} \, \right| \, 
\prod_{\alpha \neq 0 \in V^*} \tr_E(\pi_{v_\alpha}) \leq \left(
c_K \cdot 2^d \cdot p^3 
\cdot (\log N_E + \log p + d)\right)^{2^d - 1},\]
where $c_K$ is the 
effective constant implicit in equation~(\ref{eq:normtrace}).
Taking the $2^{d - 1}$st root of both sides yields the desired conclusion.
\end{proof}

\section{The Ineffective Bound}

\begin{lm} \label{inun}
If $p$ is the smallest acceptable unexceptional prime for an elliptic curve
$E$ without CM, then
$p \lllk 1$.
\end{lm}
\begin{proof}
By Serre's Open Image Theorem \cite{serre},
it suffices to verify the statement for all but finitely
many elliptic curves $E$ over $K$.
In order to do this, let $p$ be some acceptable prime.
In particular, $p \geq 23$, so
the genera of the modular curves
$X_0(p)$, $X_{\text{split}}(p)$, and $X_{\text{nonsplit}}(p)$
are all at least $2$.
By Falting's theorem \cite{faltings}, there are finitely
many points on each of these modular curves,
which completes the proof.
\end{proof}

\begin{thm} \label{ineffective}
Let $E$ be an elliptic curve over a number field $K$ without CM.
Then any exceptional prime $\ell$ satisfies
\[\ell \lllk \log N_E.\]
Moreover, the product of all exceptional primes satisfies
\[\prod \ell \lllk 4^{a_E} \cdot (\log N_E)^{14}.\]
\end{thm}
\begin{proof}
This is an immediate consequence of Lemmas
\ref{lm:redone}, \ref{vexc}, and \ref{inun}.
\end{proof}

\section{The Effective Bound}
The bound on the smallest unexceptional prime $p$ in the
previous section relies on Falting's theorem, which at
the moment is ineffective. Here we give an effective bound
on $p$ (which depends on the curve $E$, but quite gently),
using the results of Section \ref{sec:ub}.

\begin{lm} \label{boot}
Let $S$ be a finite set of primes,
$p$ be the smallest acceptable prime number not in $S$,
and $b$ be a constant depending only on $K$. Then for any $A$,
\[\prod_{\ell \in S} \ell \llk A \cdot p^b 
\quad \Rightarrow \quad p \llk \log A.\]
\end{lm}

\begin{proof}
Since the product of all unacceptable primes depends only on $K$,
it suffices to prove this lemma in the case that $S$ contains
all of the unacceptable primes. Using (an effective version of)
the prime number theorem,
\[p \ll \sum_{\ell < p} \log \ell \leq \log 
\left(\prod_{\ell \in S} \ell\right) \llk \log 
\left(A \cdot p^b\right) \llk \log A + \log p.\]
The desired result follows immediately.
\end{proof}

\begin{thm} \label{effective}
Let $E$ be an elliptic curve over a number field $K$ without CM.
Then any exceptional prime $\ell$ satisfies
\[\ell \llk \log N_E \cdot (\log \log N_E)^3.\]
Moreover, the product of all exceptional primes satisfies
\[\prod \ell \llk 4^{a_E} \cdot (\log N_E)^{14} \cdot (a_E + \log \log N_E)^6 \cdot (\log \log N_E)^{36}.\]
\end{thm}
\begin{proof}
From the bound on the product in
Lemma~\ref{lm:redone}, together with Lemma~\ref{boot},
we conclude that the smallest prime $p$ not dividing $R_E$
satisfies $p \llk \log \log N_E$.
Similarly, from the bound on the product
in Lemma~\ref{vexc}, together with Lemma~\ref{boot},
we conclude that the smallest prime $p$
that is not $V$-exceptional satisfies
$p \llk d + \log \log N_E$ (where $d = \dim V$).

Thus, Lemmas \ref{lm:redone} and \ref{vexc} imply the desired result.
\end{proof}

\section{Explicit Constants}
In this section, we estimate the dependence
on $K$ in Theorem~\ref{ithm:effective}. 
Everything used to prove Theorem~\ref{ithm:effective}
boils down to the effective
Chebotarev theorem (for which the $K$-dependence is
explicit), and Theorem~\ref{paperonethm}. To make Theorem~\ref{paperonethm}
effective, we can use the following result:

\begin{thm}
In Theorem~\ref{paperonethm}, every $\ell \in S_K$ satisfies:
\begin{align*}
\ell &\ll \exp\left(c^{n_K} \cdot (R_K \cdot n_K^{r_K} + h_K \cdot \log \Delta_K)\right); \\
\intertext{moreover, the product of all $\ell \in S_K$ is bounded by:}
\prod \ell &\ll \exp\left(c^{n_K} \cdot (R_K \cdot n_K^{r_K} + h_K^2 \cdot (\log \Delta_K)^2)\right).
\end{align*}
Here, $c$ is an effectively computable absolute constant.
\end{thm}
\begin{proof}
See Theorem~7.9 of \cite{ourfirstpaper} for the bound
on the product of all $\ell \in S_K$. The bound on the largest
element of $S_K$ can be proved in a similar way (just replace
$B_{\text{poss}}(K, g, V)$ by $1$ in the proof of Theorem~7.9
in \cite{ourfirstpaper}).
\end{proof}

\begin{thm} \label{explicitthm}
Let $E$ be an elliptic curve over a number field $K$ without CM.
Then any exceptional prime $\ell$ satisfies
\[\ell \ll \log N_E \cdot (\log \log N_E)^3 + \exp\left(c^{n_K} \cdot (R_K \cdot n_K^{r_K} + h_K \cdot \log \Delta_K)\right).\]
Moreover, the product of all exceptional primes satisfies
\begin{align*}
\prod \ell &\ll 4^{a_E} \cdot (\log N_E)^{13} \cdot (a_E + \log \log N_E)^3 \cdot (\log \log N_E)^{36} \\
&\qquad\qquad
\cdot \exp\left(c^{n_K} \cdot (R_K \cdot n_K^{r_K} + h_K^2 \cdot (\log \Delta_K)^2)\right).
\end{align*}
Here, the constant $c$ and the constants implied by the $\ll$ symbol
are all \emph{absolute} and effectively computable.
\end{thm}
\begin{proof}

This follows from carefully keeping track of the contributions
depending on $K$ in the proof of Theorem~\ref{ithm:effective}.
It is easy to see that the contributions from the bounds given
on the set $S_K$ dominate all other contributions coming from the field $K$.
\end{proof}

\bibliography{surj-bib}{}
\bibliographystyle{plain}

\end{document}